\documentclass[11pt,reqno]{amsart}

\usepackage{amscd,amssymb,amsmath,amsthm}
\usepackage[arrow,matrix]{xy}
\usepackage{graphicx}
\usepackage{epstopdf}
\usepackage{color}
\newtheorem{thm}{Theorem}
\newtheorem{defn}{Definition}

\newtheorem{pro}{Proposition}
\newtheorem{rk}{Remark}

\newtheorem{ex}{Example}

\numberwithin{equation}{section} 

\begin{document}
\title [Classification in a rotational flow of two-dimensional   algebras ]
{Classification in a rotational flow of  two-dimensional  algebras}

\author{  U.A. Rozikov, M.V. Velasco, B.A. Narkuziev }

\address{ U.Rozikov$^{a,b,c}$\begin{itemize}
 \item[$^a$] V.I.Romanovskiy Institute of Mathematics of Uzbek Academy of Sciences;
\item[$^b$] Central Asian University, 264, Milliy bog St, 111221, Tashkent,  Uzbekistan;
\item[$^c$] Faculty of Mathematics, National University of Uzbekistan.
\end{itemize}}
\email{rozikovu@yandex.ru}

\address{ M. V. Velasco, Departamento de An\'{a}lisis Matem\'{a}tico, Facultad de Ciencias, Universidad de Granada, 18071 Granada, Spain. }
\email { vvelasco@ugr.es}

\address{B.A. Narkuziev,$^{a,b}$\begin{itemize}
 \item[$^a$] Kimyo International University in Tashkent branch Samarkand, 140143, Samarkand, Uzbekistan;
\item[$^b$] Samarkand State University, 140104, University blv. 15, Samarkand, Uzbekistan.
\end{itemize}}
 \email { bnarkuziev86@gmail.com}

\begin{abstract}
In this paper, we examine a time-dependent family of two-dimensional algebras. We investigate the conditions under which any two algebras from this family, formed at different times, are isomorphic. Our findings reveal that the flow comprises of uncountable pairwise non-isomorphic algebras, including one commutative algebra. Additionally, we compare our results with a previously established classification of 2-dimensional real algebras.
\end{abstract}


 \maketitle

\section{\textbf{Introduction}}

The paper \cite{Flow} introduces a concept of time-dependent family of finite-dimensional algebras, called as a ``flow of algebras". This flow can be seen as a continuous-time dynamical system, where the states are finite-dimensional algebras with cubic matrices of structural constants that satisfy Kolmogorov-Chapman equation (KCE) (see \cite{Flow,Cons flow,Tian}).  There are several kinds of multiplications between cubic matrices  \cite{3-Lie,Mcub}.
Therefore one has to fix a multiplication  and then consider the KCE with respect to the fixed multiplication \cite{Cons flow}.
A flow of algebras is two-parametric family, which is a generalization of (one-parametric) deformation of algebras motivated in physics (see for example \cite{Ko}) and moduli theory \cite{Ha}.

In \cite{Markov}
 Markov processes of cubic matrices are studied which are a two-parametric family of cubic stochastic matrices  satisfying the KCE.
 In the book \cite{Population}  several dynamical systems of biological models such as dynamics generated by Markov processes of cubic stochastic matrices;
  dynamics of sex-linked population; dynamical systems generated by a gonosomal evolution operator; dynamical system and an evolution algebra of mosquito population;
 and ocean ecosystems were studied. This book also gives motivations of such investigations.

In \cite{chain}  the notion of chain of evolution algebras is introduced that is a particular case of the flow of algebras.
Classification in chains of two-dimensional and three-dimensional  evolution algebras are given in recent
papers \cite{Murodov,NR}.

In this paper we consider a flow of two-dimensional algebras found in \cite{Flow}.
 We find condition (depending on time) for two  algebras of this flow to be isomorphic  at different times.
Moreover, we show that the flow comprises of uncountable pairwise non-isomorphic algebras, including one commutative algebra. We compare our results with known classification of 2-dimensional real algebras.

\subsection{Cubic matrices}

\

\

Following \cite{3-Lie,ACM,Mcub,IP} recall the notion of cubic matrix and different associative multiplication rules of cubic matrices:
a cubic matrix $Q=(q_{ijk})_{i,j,k=1}^{m}$ is a $m^{3}$-dimensional vector which can be uniquely written as
$$Q=\sum_{i,j,k=1}^{m}q_{ijk}E_{ijk},$$
where $E_{ijk}$ denotes the cubic unit (basis) matrix, i.e. $E_{ijk}$ is a $m^{3}$- cubic matrix whose $(i,j,k)$th entry is
equal to 1 and all other entries are equal to 0.

Denote by $\mathfrak{C}$ the set of all cubic matrices over a field $\mathbb{F}$. Then $\mathfrak{C}$ is an $m^{3}$-dimensional
vector space over $\mathbb{F}$, i.e. for any matrices $ A=(a_{ijk})_{i,j,k=1}^{m}, \ B=(b_{ijk})_{i,j,k=1}^{m}\in \mathfrak{C},\ \lambda\in \mathbb{F}, $
we have
$$
A+B:=(a_{ijk} + b_{ijk})_{i,j,k=1}^{m}\in \mathfrak{C}, \ \ \ \  \lambda A:=(\lambda a_{ijk})_{i,j,k=1}^{m}\in \mathfrak{C}.
$$

In \cite{ACM} (see also \cite{Cons flow})  some simple versions of multiplications between cubic matrices are given.
Denote $I=\{ 1,2,...,m\}$. Following \cite{Mcub} we define the following multiplications for basis matrices $E_{ijk}$:
\begin{equation}\label{1.1}
E_{ijk} \ast_{a} E_{lnr}=\delta_{kl}E_{ia(j,n)r},
\end{equation}
where $a:I \times I\rightarrow I, \ (j,n)\mapsto a(j,n)\in I$, is an arbitrary associative binary operation
and $\delta_{kl}$ is the Kronecker symbol.

\subsection{Two-dimensional algebras over $\mathbb{R}$}

\

\

The classification problem of finite dimensional algebras is important in algebra. In \cite{Bekbaev} author considered such problem
for two-dimensional algebras over the field of real numbers $\mathbb{R}$ and gave classifications of two-dimensional general, commutative, commutative Jordan,
 division and evolution real algebras. In \cite{ABR} authors considered the problem over algebraically closed fields.

 Now we use some notions provided in \cite{Bekbaev}. Let $\mathbb{A}$  be any 2-dimensional algebra over a field  $\mathbb{F}$
 and $e=\{e_1,e_2\}$ be a basis for $\mathbb{A}$. If $e_i\cdot e_j=A^1_{i,j}e_1 + A^2_{i,j}e_2$, $i,j=1,2$ is a multiplication table for $e$ then
 the matrix of structural constants of $\mathbb{A}$ with respect to the basis $e$ is as follows:
 \begin{equation}\label{1.2}
 \mathcal{A}=\left(\begin{array}{cccc}
 A_{1,1}^1&    A_{1,2}^1&   A_{2,1}^1&  A_{2,2}^1 \\[2mm]
 A_{1,1}^2&    A_{1,2}^2&   A_{2,1}^2&  A_{2,2}^2
\end{array} \right).
 \end{equation}
 Let $\mathbb{F}=\mathbb{R}$ and for the simplicity, we use the notation
 $$
 \mathcal{A}=\left (\begin{array}{cccc}
 \alpha_1&    \alpha_2&   \alpha_3&  \alpha_4 \\[2mm]
 \beta_1&    \beta_2&   \beta_3&  \beta_4
\end{array} \right),
 $$
 where $\alpha_1,    \alpha_2,   \alpha_3, \alpha_4, \beta_1,    \beta_2,   \beta_3,  \beta_4$ stand for any elements of $\mathbb{R}$.
\begin{thm}\label{thm1}
\cite{Bekbaev} Any non-trivial 2-dimensional real algebra is isomorphic to only one of the following listed,
by their matrices of structural constants, algebras:
 $$
 \mathcal{A}_{1}(\textbf{c})=\left (\begin{array}{cccc}
 \alpha_1&    \alpha_2&   \alpha_2 + 1&  \alpha_4 \\[2mm]
 \beta_1&    -\alpha_1&   -\alpha_1 +1&  -\alpha_2
\end{array} \right),  \ where \ \textbf{c}=(\alpha_1, \alpha_2, \alpha_4, \beta_1)\in \mathbb{R}^4,
 $$
 $$
 \mathcal{A}_{2}(\textbf{c})=\left (\begin{array}{cccc}
 \alpha_1&    0&   0&  1 \\[2mm]
 \beta_1&    \beta_2&  1 -\alpha_1 &  0
\end{array} \right),  \ where \  \beta_1\geq 0, \ \textbf{c}=(\alpha_1,  \beta_1, \beta_2)\in \mathbb{R}^3,
 $$
 $$
 \mathcal{A}_{3}(\textbf{c})=\left (\begin{array}{cccc}
 \alpha_1&    0&   0&  -1 \\[2mm]
 \beta_1&    \beta_2&  1 -\alpha_1 &  0
\end{array} \right),  \ where \  \beta_1\geq 0, \ \textbf{c}=(\alpha_1,  \beta_1, \beta_2)\in \mathbb{R}^3,
 $$
 $$
 \mathcal{A}_{4}(\textbf{c})=\left (\begin{array}{cccc}
 0&    1&   1&  0 \\[2mm]
 \beta_1&    \beta_2&  1 &  -1
\end{array} \right),  \ where \   \textbf{c}=( \beta_1, \beta_2)\in \mathbb{R}^2,
 $$
 $$
 \mathcal{A}_{5}(\textbf{c})=\left (\begin{array}{cccc}
 \alpha_1&    0&   0&  0 \\[2mm]
 0&    \beta_2&  1-\alpha_1 &  0
\end{array} \right),  \ where \   \textbf{c}=(\alpha_1, \beta_2)\in \mathbb{R}^2,
 $$
 $$
 \mathcal{A}_{6}(\textbf{c})=\left (\begin{array}{cccc}
 \alpha_1&    0&   0&  0 \\[2mm]
 1&    2\alpha_1-1&  1-\alpha_1 &  0
\end{array} \right),  \ where \   \textbf{c}=\alpha_1\in \mathbb{R},
 $$
 $$
 \mathcal{A}_{7}(\textbf{c})=\left (\begin{array}{cccc}
 \alpha_1&    0&   0&  1 \\[2mm]
 \beta_1&    1-\alpha_1&  -\alpha_1 &  0
\end{array} \right),  \ where \ \beta_1\geq 0, \textbf{c}=(\alpha_1, \beta_1)\in \mathbb{R}^2,
 $$
 $$
 \mathcal{A}_{8}(\textbf{c})=\left (\begin{array}{cccc}
 \alpha_1&    0&   0&  -1 \\[2mm]
 \beta_1&    1-\alpha_1&  -\alpha_1 &  0
\end{array} \right),  \ where \ \beta_1\geq 0, \textbf{c}=(\alpha_1, \beta_1)\in \mathbb{R}^2,
 $$
 $$
 \mathcal{A}_{9}(\textbf{c})=\left (\begin{array}{cccc}
 0&    1&   1&  0 \\[2mm]
 \beta_1&    1&   0&  -1
\end{array} \right),  \ where \   \textbf{c}=\beta_1 \in \mathbb{R},
 $$
 $$
 \mathcal{A}_{10}(\textbf{c})=\left (\begin{array}{cccc}
 \alpha_1&    0&   0&  0 \\[2mm]
 0&    1-\alpha_1&  -\alpha_1 &  0
\end{array} \right),  \ where \  \textbf{c}=\alpha_1 \in \mathbb{R},
 $$
 $$
 \mathcal{A}_{11}=\left (\begin{array}{cccc}
 \frac{1}{3}&    0&   0&  0 \\[2mm]
 1&    \frac{2}{3}&  -\frac{1}{3} &  0
\end{array} \right), \ \
\mathcal{A}_{12}=\left (\begin{array}{cccc}
 0&    1&   1&  0 \\[2mm]
 1&    0&  0 &  -1
\end{array} \right),
$$
$
\mathcal{A}_{13}=\left (\begin{array}{cccc}
 0&    1&   1&  0 \\[2mm]
 -1&    0&  0 &  -1
\end{array} \right), \
\mathcal{A}_{14}=\left (\begin{array}{cccc}
 0&    1&   1&  0 \\[2mm]
 0&    0&  0 &  -1
\end{array} \right), \
 \mathcal{A}_{15}=\left (\begin{array}{cccc}
 0&    0&   0&  0 \\[2mm]
 1&    0&  0 &  0
\end{array} \right).
$
\end{thm}

As shows the following example, two algebras with different parameter values, belonging to the same class as defined in Theorem \ref{thm1} may not be isomorphic.

\begin{ex} We consider the algebra $\mathcal{A}_{1}(\textbf{c})$ with parameters  $\textbf{c}$ as follows $\alpha_{1}=\alpha_{2}=\alpha_{4}=\beta_{1}=0$
and $\alpha_{1}=\alpha_{2}=\alpha_{4}=0,\ \beta_{1}=1$, then we have the following algebras
$$
 \mathcal{A}_{1}(0,0,0,0)=\left (\begin{array}{cccc}
 0&    0&    1&  0 \\[2mm]
 0&    0&   1&  0
\end{array} \right), \ \
\mathcal{A}_{1}(0,0,0,1)=\left (\begin{array}{cccc}
 0&    0&    1&  0 \\[2mm]
 1&    0&   1&  0
\end{array} \right).
 $$
 The multiplication tables for these algebras are as follows
$$
\mathcal{A}_{1}(0,0,0,0):\ e_{1}e_{1}=0, \ e_{1}e_{2}=0, \ e_{2}e_{1}=e_{1}+e_{2},\ e_{2}e_{2}=0;
$$
$$
\mathcal{A}_{1}(0,0,0,1):\ \widetilde{e}_{1}\widetilde{e}_{1}=\widetilde{e}_{2}, \ \widetilde{e}_{1}\widetilde{e}_{2}=0,\
 \widetilde{e}_{2}\widetilde{e}_{1}=\widetilde{e}_{1} + \widetilde{e}_{2},\ \widetilde{e}_{2}\widetilde{e}_{2}=0.
$$
Assume that $\mathcal{A}_{1}(0,0,0,0)$ and $ \mathcal{A}_{1}(0,0,0,1)$ are isomorphic. Then there exists a change of basis as follows
$
\left\{\begin{array}{cccc}
\widetilde{e}_{1}=x_{1}e_{1} + x_{2}e_{2}\\
\widetilde{e}_{2}=y_{1}e_{1} + y_{2}e_{2}
\end{array}\right.
$
 where $x_{1}, x_{2}, y_{1}, y_{2}\in \mathbb{R}$  and $x_{1}y_{2}\neq x_{2}y_{1}$. Hence,
 $$
 \widetilde{e}_{1}\widetilde{e}_{1}=(x_{1}e_{1} + x_{2}e_{2})(x_{1}e_{1} + x_{2}e_{2})= x_{2}x_{1}e_{2}e_{1}=x_{2}x_{1}(e_{1}+e_{2}).
 $$
On the other hand $\widetilde{e}_{1}\widetilde{e}_{1}=\widetilde{e}_{2}=y_{1}e_{1} + y_{2}e_{2}$.
It follows that $y_{1}=y_{2}=x_{2}x_{1}$. Likewise,
$
\widetilde{e}_{2}\widetilde{e}_{1}=y_{2}x_{1}e_{2}e_{1}=y_{2}x_{1}(e_{1}+e_{2})
$
and $\widetilde{e}_{2}\widetilde{e}_{1}=\widetilde{e}_{1} + \widetilde{e}_{2}=(x_{1}+y_{1})e_{1} + (x_{2}+y_{2})e_{2} $.
Then we have $ x_{1}+y_{1}=x_{2}+y_{2}=y_{2}x_{1}$, it means that $x_{1}=x_{2}$. It is a contradiction to $x_{1}y_{2}\neq x_{2}y_{1}$.
\end{ex}
Thus we have the following
\begin{rk}
The parametric algebras in a given class of Theorem \ref{thm1} may not be isomorphic for different parameters.
\end{rk}

\subsection{Flow of algebras}

\

\

Following \cite{Flow} we define a notion of flow of algebras (FA).
Consider a family
$$\{A^{[s,t]}: s,t\in \mathbb{R},\  0\leq s\leq t\}$$ of arbitrary
$m$-dimensional algebras over the field $ \mathbb{R}$, with  basis $e_{1},e_{2},\dots,e_{m}$
and  multiplication table
\begin{equation} e_{i}e_{j}=\sum_{k=1}^{m}c_{ijk}^{[s,t]}e_{k}, \ \ \ \ \ \ i,j=1, \dots , m.\end{equation}
Here parameters $s,t$ are considered as time.

Denote by $\mathcal M^{[s,t]}=\left(c_{ijk}^{[s,t]}\right)_{i,j,k=1,\dots m}$ the matrix of structural constants of  $ A^{[s,t]}$.
\begin{defn}
Fix an arbitrary multiplication of cubic matrices, say $\ast_{\mu}$.

 A family $\{A^{[s,t]}: s,t\in \mathbb{R},\  0\leq s\leq t\}$ of $m$-dimensional  algebras over the field $\mathbb R$
is called an flow of algebras (FA) of type $\mu$ if the matrices $\mathcal M^{[s,t]}$ of structural constants satisfy the Kolmogorov-Chapman equation
 (for cubic matrices):
\begin{equation}\label{Chapman-Kolmogorov}
\mathcal M^{[s,t]}=\mathcal M^{[s,\tau]} \ast_{\mu} \mathcal M^{[\tau,t]}, \,\,\,for \,\,\,all \,\,\,0\leq s<\tau<t.
\end{equation}
\end{defn}
\begin{defn}
An FA is called a time-homogeneous FA if the matrix $\mathcal M^{[s,t]}$ depends only on $t-s$. In this case we write $\mathcal M^{[t-s]}$.
\end{defn}
\begin{defn}
An FA is called periodic if its matrix $\mathcal M^{[s,t]}$ is periodic with respect to at
least one of the variables $s,\ t,$ i.e. (periodicity with respect to t) $\mathcal M^{[s,t+P]}=\mathcal M^{[s,t]}$
for all values of $t$. The constant $P$ is called the period, and is required to be nonzero.
\end{defn}

\begin{rk}
Following  \cite{Flow}	we give a comparison of FA with Deformations of Algebras: In the case of finite-dimensional algebras, a deformation refers to the transformation of a given algebra into another algebra with the same dimension \cite{Gerstenhaber}. This is equivalent to changing the multiplication table (structural constants) of the given algebra to another one. Often, the algebra at $t=0$ is considered as the initial algebra, and the algebra at the current time $t$ is considered as the current algebra. In \cite{Gerstenhaber} (see also \cite{Fox}) deformations of an algebra $\mathrm{A}$  are given by a bilinear function $f_t$. That is, one considers the algebra $\mathrm{A}_t$ with multiplication $f_t$ as the generic element of a ``one-parameter family of deformations of $\mathrm{A}$". Thus, deformations of an algebra are given by the rule $f_t$, which has an explicit form. Similarly, a flow of algebras is given by $\mathcal M^{[s,t]}$ with the rule (\ref{Chapman-Kolmogorov}).  Some relations of an FA can be also seen in the following example: a plane deformation \cite{Ogden}
is restricted to the plane described by the basis vectors $e_{1}, e_{2}$. It is known that the deformation gradient of this plane deformation has the form
$$
\mathbf{F}=\left (\begin{array}{cccc}
 \cos\theta&    \sin\theta&   0& \\[2mm]
 -\sin\theta&    \cos\theta&  0 & \\[2mm]
 0&    0&  1 &
\end{array} \right)
\left (\begin{array}{cccc}
 \lambda_{1}&    0&   0& \\[2mm]
 0&    \lambda_{2}&  0& \\[2mm]
 0&    0&  1 &
\end{array} \right),
$$
where $\theta$ is the angle of rotation and $\lambda_{1}, \ \lambda_{2}$ are the principal stretches. Comparing this matrix with matrix
(\ref{2.4}) (see below) one can see that the corresponding FA and the plane deformation have similar matrices.
\end{rk}

To construct an FA of type $\mu$ one has to solve (\ref{Chapman-Kolmogorov}).
In this paper we consider one its solution for the case $m=2$ and a fixed multiplication of cubic matrices called type $C$.

\section{\textbf{Two-dimensional flow of algebras}}
Let $m=2$ and consider  two-dimensional flow of  algebras  (see \cite{Flow}) which defined by multiplication:
\begin{equation}\label{multC}
E_{ijk} \cdot E_{lnr}=\delta_{kl}\delta_{jn}E_{ijr},
\end{equation}
where $\delta_{ij}$- Kronecker symbol.

The multiplication (\ref{multC}) and corresponding FA are called type $C$ in \cite{Flow}.
Extending this multiplication to arbitrary cubic matrices
$$ A=(a_{ijk})_{i,j,k=1}^{m}, \ \ B=(b_{ijk})_{i,j,k=1}^{m}, \ \ C=(c_{ijk})_{i,j,k=1}^{m},$$
we get that the entries of $C=AB$ can be written as
$$
c_{ijr}=\sum_{k=1}^{m}a_{ijk}b_{kjr}.
$$
Write cubic matrix $\mathcal M^{[s,t]}$ for $m=2$ in the following convenient form:
\begin{equation}\label{2.2}
\mathcal M^{[s,t]}=\left (\begin{array}{cccc}
 c_{111}^{[s,t]}&    c_{112}^{[s,t]} \\[2mm]
 c_{121}^{[s,t]}&   c_{122}^{[s,t]}
\end{array} \ \ \Bigg| \ \
 \begin{array}{cccc}
c_{211}^{[s,t]} &   c_{212}^{[s,t]}\\[2mm]
c_{221}^{[s,t]}&    c_{222}^{[s,t]}
\end{array} \right).
\end{equation}
In this case equation (\ref{Chapman-Kolmogorov}) has the following form:
$$ c_{ijr}^{[s,t]}=c_{ij1}^{[s,\tau]}c_{1jr}^{[\tau,t]} + c_{ij2}^{[s,\tau]}c_{2jr}^{[\tau,t]}, \ \ \ \ \ \ i,j,r=1,2.$$
This is a quadratic system of eight equations with eight unknown functions $c_{ijr}^{[s,t]} $ of two
variables $s,t, \ 0\leq s<t$. If we consider four equations with $j = 1$, the unknowns in them do not participate
in the other four equations with $j = 2$. Therefore the equations for $j=1$ and $j=2$
are independent. Hence it suffices to solve the system only for $j=1$. Denote
$a_{ir}^{[s,t]}=c_{i1r}^{[s,t]}$ and we have the following system:
\begin{equation}\label{2.3}
\left\{\begin{array}{cccc}
a_{11}^{[s,t]}=a_{11}^{[s,\tau]}a_{11}^{[\tau, t]} + a_{12}^{[s,\tau]}a_{21}^{[\tau, t]}\\[2mm]
a_{12}^{[s,t]}=a_{11}^{[s,\tau]}a_{12}^{[\tau, t]} + a_{12}^{[s,\tau]}a_{22}^{[\tau, t]}\\[2mm]
a_{21}^{[s,t]}=a_{21}^{[s,\tau]}a_{11}^{[\tau, t]} + a_{22}^{[s,\tau]}a_{21}^{[\tau, t]}\\[2mm]
a_{22}^{[s,t]}=a_{21}^{[s,\tau]}a_{12}^{[\tau, t]} + a_{22}^{[s,\tau]}a_{22}^{[\tau, t]}.
\end{array}\right.
\end{equation}
The full set of solutions to the system (\ref{2.3}) is not known yet. But there is a very wide class
of its solutions see \cite{chain,ImVelasco,realchains,RM}.
One of these known solutions is the following matrix:
\begin{equation}\label{2.4}
\left (\begin{array}{cccc}
a_{11}^{[s,t]}&   a_{12}^{[s,t]}\\[2mm]
a_{21}^{[s,t]}&   a_{22}^{[s,t]}
\end{array}\right)=
\left (\begin{array}{cccc}
  \cos(t-s)&  \sin(t-s)\\[2mm]
-\sin(t-s)&  \cos(t-s)
\end{array}\right).
\end{equation}

\begin{rk}
Note that the matrix (\ref{2.4}) defines a linear rotation flow \cite{Linflow,Ogden}.
\end{rk}

By $(\overline{a}_{ij}^{[s,t]})$ we denote the matrix transposed to matrix $(a_{ij}^{[s,t]})$. Any pair $(a_{ij}^{[s,t]})$, $(\overline{a}_{ij}^{[s,t]})$
 of solutions of the system (\ref{2.3}) generates an FA
of type $C$ corresponding to the matrix $\mathcal M^{[s,t]}$ with entries
$$ c_{i1r}^{[s,t]}=a_{ir}^{[s,t]}, \ \ \ \ \ c_{i2r}^{[s,t]}=\overline{a}_{ir}^{[s,t]} . $$
For the matrix (\ref{2.4}) we get the following cubic matrix, which  generates an FA of type $C$:
\begin{equation}\label{2.5}
\mathcal M^{[s,t]}=\left (\begin{array}{cccc}
 \cos(t-s)& \sin(t-s)\\[2mm]
 \cos(t-s)&  -\sin(t-s)
\end{array} \Bigg|
 \begin{array}{cccc}
-\sin(t-s)&  \cos(t-s)\\[2mm]
  \sin(t-s)&  \cos(t-s)
\end{array} \right).
\end{equation}
We note that a cubic matrix $(c_{ijk})_{i,j,k=1}^{m}$ generates a commutative algebra iff $c_{ijk}=c_{jik}$.
Hence the algebra $ A^{[s,t]}$ of the FA corresponding to the cubic matrix (\ref{2.5}) is
commutative iff $\cos(t-s)=-\sin(t-s)$, i.e. $t=s+\frac{3\pi}{4}+\pi n$, $n=0,1,2,....$
Therefore this FA  is almost non-commutative (with respect to Lebesgue measure on $\mathcal{T}=\{(s,t): 0\leq s\leq t\}$).

 Note that this FA is a time-homogeneous, therefore we can consider it
with respect to one time parameter $\widehat{t}=t-s$. So for convenience we write $ A^{[t]}$ instead of $ A^{[s,t]}$.
\begin{rk}
The FA  (\ref{2.5}) is periodic and its period is $2\pi$.
\end{rk}

\textbf {The main problem} of this paper is to classify algebras of $ A^{[t]}$, i.e. at different values of time FA
forms various algebras, we need to find the time-dependent conditions under which algebras are isomorphic.

Let $\{e_{1}, e_{2}\}$ and $\{e_{1}', e_{2}'\}$ are bases of algebras $ A^{[t_{1}]}$ and $ A^{[t_{2}]}$ respectively.
We consider the following multiplication table of basis $\{e_{1}, e_{2}\}$ corresponding to (\ref{2.5}):
\begin{equation}\label{2.6}
\begin{array}{cccc}
e_{1}e_{1}=c_{111}e_{1} + c_{112}e_{2}= \cos t_{1}e_{1} + \sin t_{1}e_{2},\\[1mm]
e_{1}e_{2}=c_{121}e_{1} + c_{122}e_{2}= \cos t_{1}e_{1} - \sin t_{1}e_{2},\\[1mm]
 \ \ e_{2}e_{1}=c_{211}e_{1} + c_{212}e_{2}= -\sin t_{1}e_{1} + \cos t_{1}e_{2},\\[1mm]
e_{2}e_{2}=c_{221}e_{1} + c_{222}e_{2}= \sin t_{1}e_{1} + \cos t_{1}e_{2}.
\end{array}
\end{equation}
Assume that $ A^{[t_{1}]}$ and $ A^{[t_{2}]}$ are isomorphic. Then there exists  a change of basis as follows
\begin{equation}\label{2.7}
\left\{\begin{array}{cccc}
e_{1}'=x_{1}e_{1} + x_{2}e_{2}\\
e_{2}'=y_{1}e_{1} + y_{2}e_{2},
\end{array}\right.
\end{equation}
where $x_{1}, x_{2}, y_{1}, y_{2}\in \mathbb{R} $ and $x_{1}y_{2}\neq x_{2}y_{1}$.

 We can also write a multiplication table similar to (\ref{2.6}) for $\{e_{1}', e_{2}'\}$. On other hand from
(\ref{2.7}) we have
$$ \begin{array}{cccc}
e_{1}'e_{1}'=x_{1}^{2}e_{1}e_{1} + x_{1}x_{2}e_{1}e_{2} + x_{2}x_{1}e_{2}e_{1} + x_{2}^{2}e_{2}e_{2} ,\\[1mm]
 \ \ \ \ \ e_{1}'e_{2}'=x_{1}y_{1}e_{1}e_{1} + x_{1}y_{2}e_{1}e_{2} + x_{2}y_{1}e_{2}e_{1} + x_{2}y_{2}e_{2}e_{2},\\[1mm]
 \ \ \ \ \ e_{2}'e_{1}'=x_{1}y_{1}e_{1}e_{1} + x_{2}y_{1}e_{1}e_{2} + x_{1}y_{2}e_{2}e_{1} + x_{2}y_{2}e_{2}e_{2},\\[1mm]
e_{2}'e_{2}'=y_{1}^{2}e_{1}e_{1} + y_{1}y_{2}e_{1}e_{2} + y_{2}y_{1}e_{2}e_{1} + y_{2}^{2}e_{2}e_{2}.
\end{array}$$
If we use (\ref{2.6}) then  after some easy calculations we get
\begin{equation}\label{2.8}
\begin{array}{cccc}
e_{1}'e_{1}'=(x_{1}^{2}\cos t_{1} + x_{1}x_{2}(\cos t_{1}-\sin t_{1}) + x_{2}^{2}\sin t_{1})e_{1} +\\[1mm]
 + (x_{1}^{2}\sin t_{1} + x_{1}x_{2}(\cos t_{1}-\sin t_{1}) + x_{2}^{2}\cos t_{1})e_{2}, \\[1mm]
 \ \ \ \ e_{1}'e_{2}'=(x_{1}y_{1}\cos t_{1} + x_{1}y_{2}\cos t_{1} - x_{2}y_{1}\sin t_{1} + x_{2}y_{2}\sin t_{1} )e_{1} + \\[1mm]
 + (x_{1}y_{1}\sin t_{1} - x_{1}y_{2}\sin t_{1} + x_{2}y_{1}\cos t_{1} + x_{2}y_{2}\cos t_{1} )e_{2} ,\\[1mm]
 \ \ \ \ e_{2}'e_{1}'=(x_{1}y_{1}\cos t_{1} + x_{2}y_{1}\cos t_{1} - x_{1}y_{2}\sin t_{1} + x_{2}y_{2}\sin t_{1} )e_{1} + \\[1mm]
 + (x_{1}y_{1}\sin t_{1} - x_{2}y_{1}\sin t_{1} + x_{1}y_{2}\cos t_{1} + x_{2}y_{2}\cos t_{1} )e_{2} ,\\[1mm]
 e_{2}'e_{2}'=(y_{1}^{2}\cos t_{1} + y_{1}y_{2}(\cos t_{1}-\sin t_{1}) + y_{2}^{2}\sin t_{1})e_{1} +\\[1mm]
 + (y_{1}^{2}\sin t_{1} + y_{1}y_{2}(\cos t_{1}-\sin t_{1}) + y_{2}^{2}\cos t_{1})e_{2}.
\end{array}
\end{equation}
If we use the multiplication table of $\{e_{1}', e_{2}'\}$ and the change (\ref{2.7}) then we have
\begin{equation}\label{2.9}
\begin{array}{cccc}
e_{1}'e_{1}'=(x_{1}\cos t_{2} + y_{1}\sin t_{2})e_{1} + (x_{2}\cos t_{2} + y_{2}\sin t_{2})e_{2},\\[1mm]
e_{1}'e_{2}'=(x_{1}\cos t_{2} - y_{1}\sin t_{2})e_{1} + (x_{2}\cos t_{2} - y_{2}\sin t_{2})e_{2},\\[1mm]
e_{2}'e_{1}'=(-x_{1}\sin t_{2} + y_{1}\cos t_{2})e_{1} + (-x_{2}\sin t_{2} + y_{2}\cos t_{2})e_{2},\\[1mm]
e_{2}'e_{2}'=(x_{1}\sin t_{2} + y_{1}\cos t_{2})e_{1} + (x_{2}\sin t_{2} + y_{2}\cos t_{2})e_{2}.
\end{array}
\end{equation}
By equating the coefficients of the corresponding terms in systems (\ref{2.8}) and (\ref{2.9}), we obtain the following system
$$ \left\{\begin{array}{cccc}
x_{1}^{2}\cos t_{1} + x_{1}x_{2}(\cos t_{1}-\sin t_{1}) + x_{2}^{2}\sin t_{1}=x_{1}\cos t_{2} + y_{1}\sin t_{2}\\[1mm]
x_{1}^{2}\sin t_{1} + x_{1}x_{2}(\cos t_{1}-\sin t_{1}) + x_{2}^{2}\cos t_{1} =x_{2}\cos t_{2} + y_{2}\sin t_{2}\\[1mm]
x_{1}y_{1}\cos t_{1} + x_{1}y_{2}\cos t_{1} - x_{2}y_{1}\sin t_{1} + x_{2}y_{2}\sin t_{1}= x_{1}\cos t_{2} - y_{1}\sin t_{2}\\[1mm]
x_{1}y_{1}\sin t_{1} - x_{1}y_{2}\sin t_{1} + x_{2}y_{1}\cos t_{1} + x_{2}y_{2}\cos t_{1} = x_{2}\cos t_{2} - y_{2}\sin t_{2}\\[1mm]
x_{1}y_{1}\cos t_{1} + x_{2}y_{1}\cos t_{1} - x_{1}y_{2}\sin t_{1} + x_{2}y_{2}\sin t_{1} = -x_{1}\sin t_{2} + y_{1}\cos t_{2} \\[1mm]
x_{1}y_{1}\sin t_{1} - x_{2}y_{1}\sin t_{1} + x_{1}y_{2}\cos t_{1} + x_{2}y_{2}\cos t_{1}=-x_{2}\sin t_{2} + y_{2}\cos t_{2} \\[1mm]
y_{1}^{2}\cos t_{1} + y_{1}y_{2}(\cos t_{1}-\sin t_{1}) + y_{2}^{2}\sin t_{1} =x_{1}\sin t_{2} + y_{1}\cos t_{2}\\[1mm]
y_{1}^{2}\sin t_{1} + y_{1}y_{2}(\cos t_{1}-\sin t_{1}) + y_{2}^{2}\cos t_{1}=x_{2}\sin t_{2} + y_{2}\cos t_{2}\\[1mm]
x_{1}y_{2}\neq x_{2}y_{1}.
\end{array}\right. $$
We write this system in the following convenient form
\begin{equation}\label{2.10}
\left\{\begin{array}{cccc}
(x_{1}^{2} + x_{1}x_{2})\cos t_{1} + (x_{2}^{2}-x_{1}x_{2})\sin t_{1}=x_{1}\cos t_{2} + y_{1}\sin t_{2}\\[1mm]
(x_{2}^{2} +x_{1}x_{2})\cos t_{1} + (x_{1}^{2} - x_{1}x_{2})\sin t_{1} =x_{2}\cos t_{2} + y_{2}\sin t_{2}\\[1mm]
(x_{1}y_{1} + x_{1}y_{2})\cos t_{1}  + (x_{2}y_{2} - x_{2}y_{1})\sin t_{1}= x_{1}\cos t_{2} - y_{1}\sin t_{2}\\[1mm]
(x_{2}y_{1} + x_{2}y_{2})\cos t_{1} + (x_{1}y_{1} - x_{1}y_{2})\sin t_{1} = x_{2}\cos t_{2} - y_{2}\sin t_{2}\\[1mm]
(x_{1}y_{1} + x_{2}y_{1})\cos t_{1}  + (x_{2}y_{2}- x_{1}y_{2})\sin t_{1} =  y_{1}\cos t_{2} - x_{1}\sin t_{2}  \\[1mm]
(x_{1}y_{2} + x_{2}y_{2})\cos t_{1} + (x_{1}y_{1} - x_{2}y_{1})\sin t_{1} = y_{2}\cos t_{2} -x_{2}\sin t_{2}  \\[1mm]
(y_{1}^{2} + y_{1}y_{2})\cos t_{1} + (y_{2}^{2} - y_{1}y_{2})\sin t_{1} = y_{1}\cos t_{2} + x_{1}\sin t_{2} \\[1mm]
(y_{2}^{2} +   y_{1}y_{2})\cos t_{1} + (y_{1}^{2} - y_{1}y_{2})\sin t_{1} =y_{2}\cos t_{2} + x_{2}\sin t_{2}  \\[1mm]
x_{1}y_{2}\neq x_{2}y_{1}.
\end{array}\right.
\end{equation}
For convenience we  denote
$$
x_{1} + x_{2}=u, \ \ y_{1} + y_{2}=v, \ \ x_{1} - x_{2}=\alpha, \ \ y_{1} - y_{2}=\beta.
$$
Then the system (\ref{2.10}) will be in the following form
\begin{equation}\label{2.11}
\left\{\begin{array}{cccc}
x_{1} u \cos t_{1} - x_{2}\alpha \sin t_{1}=x_{1}\cos t_{2} + y_{1}\sin t_{2}\\[1mm]
x_{2} u \cos t_{1} + x_{1}\alpha \sin t_{1} =x_{2}\cos t_{2} + y_{2}\sin t_{2}\\[1mm]
x_{1} v \cos t_{1}  - x_{2}\beta \sin t_{1}= x_{1}\cos t_{2} - y_{1}\sin t_{2}\\[1mm]
x_{2} v \cos t_{1} + x_{1}\beta \sin t_{1} = x_{2}\cos t_{2} - y_{2}\sin t_{2}\\[1mm]
y_{1} u \cos t_{1}  - y_{2}\alpha \sin t_{1} =  y_{1}\cos t_{2} - x_{1}\sin t_{2}  \\[1mm]
y_{2} u \cos t_{1} + y_{1}\alpha \sin t_{1} = y_{2}\cos t_{2} -x_{2}\sin t_{2}  \\[1mm]
y_{1} v \cos t_{1} -  y_{2}\beta \sin t_{1} = y_{1}\cos t_{2} + x_{1}\sin t_{2} \\[1mm]
y_{2} v \cos t_{1} +  y_{1}\beta \sin t_{1} =y_{2}\cos t_{2} + x_{2}\sin t_{2}  \\[1mm]
x_{1}y_{2}\neq x_{2}y_{1}.
\end{array}\right.
\end{equation}
 Note that $u^{2} + v^{2}\neq 0$ and $\alpha^{2} + \beta^{2}\neq 0$, otherwise $x_{1}y_{2}= x_{2}y_{1} $.
 By adding each of  two equations of system (\ref{2.11}), we obtain
 \begin{equation}\label{2.12}
\left\{\begin{array}{cccc}
u^{2}\cos t_{1} + \alpha^{2} \sin t_{1}=u \cos t_{2} + v \sin t_{2}\\[1mm]
uv \cos t_{1} + \alpha \beta \sin t_{1} =u \cos t_{2} - v \sin t_{2}\\[1mm]
uv \cos t_{1} + \alpha \beta \sin t_{1} =v \cos t_{2} - u \sin t_{2}\\[1mm]
v^{2}\cos t_{1} + \beta^{2} \sin t_{1}=v \cos t_{2} + u \sin t_{2}\\[1mm]
\alpha v\neq u\beta.
\end{array}\right.
\end{equation}
 If the system (\ref{2.10}) has a solution $\{x_1, x_2, y_1, y_2\}$, then the system (\ref{2.12}) also has a solution $\{u, v, \alpha, \beta\}$.
 By our assumption the system (\ref{2.10}) has a non-trivial solution, then the system (\ref{2.12}) also has a non-trivial solution.

 From the second and third equations of the  system (\ref{2.12}) we get
 $$\begin{array}{cccc}
  u \cos t_{2} - v \sin t_{2}= v \cos t_{2} - u \sin t_{2} , \\[1mm]
  u( \cos t_{2} + \sin t_{2}) - v(\cos t_{2} + \sin t_{2})=0.
 \end{array}$$
 Therefore
  \begin{equation}\label{2.13}
  (u-v)( \cos t_{2} + \sin t_{2})=0
 \end{equation}
 From the last equality, we have two cases.

 \textbf{Case 1.}  Let $u=v$ then $x_{1} + x_{2}= y_{1} + y_{2}=u$.

In this case the system (\ref{2.12}) will be in the following form
\begin{equation}\label{2.14}
\left\{\begin{array}{cccc}
u^{2}\cos t_{1} + \alpha^{2} \sin t_{1}=u \cos t_{2} + u \sin t_{2}\\[1mm]
u^{2} \cos t_{1} + \alpha \beta \sin t_{1} =u \cos t_{2} - u \sin t_{2}\\[1mm]
u^{2}\cos t_{1} + \beta^{2} \sin t_{1}=u \cos t_{2} + u \sin t_{2}\\[1mm]
\alpha v\neq u\beta.
\end{array}\right.
\end{equation}
Note that $x_{i}\neq y_{i}$, $i=1,2$ otherwise $x_{1}y_{2}= x_{2}y_{1}$.
From the first and third equations of this system we get
\begin{equation}\label{2.15}
 \alpha^{2} \sin t_{1}=\beta^{2} \sin t_{1}.
\end{equation}
\textbf{Case 1.1.} First we consider the case $\sin t_{1}=0$ then $\cos t_{1}=\pm 1.$

Let $ \left\{\begin{array}{cccc}
\sin t_{1}=0\\
\cos t_{1}=1
\end{array}\right.$, then from (\ref{2.14}) we have
$\left\{\begin{array}{cccc}
u^{2}=u \cos t_{2} + u \sin t_{2}\\
u^{2}=u \cos t_{2} - u \sin t_{2} \end{array}\right.$.\\
We know that $u\neq 0$, then
$$\left\{\begin{array}{cccc}
u=\cos t_{2} + \sin t_{2}\\
u=\cos t_{2} -  \sin t_{2} \end{array}\right..$$
From this system we get $\sin t_{2}=0$, then $\cos t_{2}=\pm 1$. In this case $u=\pm 1$.\\
Therefore $ \left\{\begin{array}{cccc}
x_{1} + x_{2}=\pm1\\[1mm]
 y_{1} + y_{2}=\pm1
\end{array}\right.$,
solutions of this systems are
\begin{equation}\label{2.16}
x_{1}=\gamma, \ \ x_{2}=\pm1-\gamma, \ \ y_{1}=\mu , \ \ y_{2}=\pm1-\mu, \ \ \gamma\neq \mu \ \ \gamma, \mu \in\mathbb{R}.
\end{equation}
The case
$ \left\{\begin{array}{cccc}
\sin t_{1}=0\\
\cos t_{1}=-1
\end{array}\right.$ is similar to the above case.

 In this case from  $\sin t_{1}=0$ it follows that $\sin t_{2}=0$ and solutions of the system (\ref{2.10}) are (\ref{2.16}).
And we have the following  isomorphic two-dimensional  algebras $A_{1}$ and $A_{-1}$ with structural constants matrix
$$
\mathcal M_{1}=\left (\begin{array}{cccc}
 1& 0\\[2mm]
 1&  0
\end{array} \Bigg|
 \begin{array}{cccc}
0&  1\\[2mm]
  0&  1
\end{array} \right) \qquad \text{and} \ \ \
\mathcal M_{-1}=\left (\begin{array}{cccc}
 -1& 0\\[2mm]
 -1&  0
\end{array} \Bigg|
 \begin{array}{cccc}
0&  -1\\[2mm]
  0&  -1
\end{array} \right)
\qquad \text{respectively.}
$$

Indeed if we take the change of basis $e_i'=-e_i, \ i=1,2$ then it is easy to see that they are isomorphic.

\textbf{Case 1.2.} Now we consider $\sin t_{1}\neq0$. From the equality (\ref{2.15}) we have $\alpha^{2}=\beta^{2}$.
It means that $\alpha=\pm \beta$ and note that $\beta \neq 0.$

\textbf{Case 1.2.1.} If $\alpha=\beta$, then $\alpha v= u\beta$. It is contradiction.

\textbf{Case 1.2.2.} Let $\alpha=-\beta$, then we have the following system
$ \left\{\begin{array}{cccc}
x_{1}-x_{2}=\alpha\\
x_{1}+x_{2}=u\\
y_{1}-y_{2}=-\alpha\\
y_{1}+y_{2}=u
\end{array}\right.$.\\
Solution of this system is $x_{1}=y_{2}=\frac{u+\alpha}{2}$, $x_{2}=y_{1}=\frac{u-\alpha}{2}$.
In this case
$$
\begin{vmatrix}
x_1& x_2\\
y_1& y_2
\end{vmatrix}=x_1y_2-x_2y_1=(\frac{u+\alpha}{2})^{2}-(\frac{u-\alpha}{2})^{2}=u\alpha \neq0,
$$
and the system (\ref{2.14}) will be
\begin{equation}\label{2.17}
\left\{\begin{array}{cccc}
u^{2}\cos t_{1} + \alpha^{2} \sin t_{1}=u \cos t_{2} + u \sin t_{2}\\[1mm]
u^{2} \cos t_{1} - \alpha^{2} \sin t_{1} =u \cos t_{2} - u \sin t_{2}\\[1mm]
x_{1}y_{2}\neq x_{2}y_{1}.
\end{array}\right.
\end{equation}
By adding the first two equations of this system we have
\begin{equation}\label{2.18} u \cos t_1= \cos t_2.
\end{equation}
\textbf{Case 1.2.2.1.}
Let $\cos t_1=0$, then $\cos t_2=0$. In this case the system (\ref{2.11}) will be
 \begin{equation}\label{2.19}
\left\{\begin{array}{cccc}
 - x_{2}\alpha \sin t_{1}= y_{1}\sin t_{2}\\[1mm]
 x_{1}\alpha \sin t_{1} = y_{2}\sin t_{2}\\[1mm]
 y_{2}\alpha \sin t_{1} =   x_{1}\sin t_{2}  \\[1mm]
 y_{1}\alpha \sin t_{1} =  -x_{2}\sin t_{2}  \\[1mm]
x_{1}y_{2}\neq x_{2}y_{1}.
\end{array}\right.
\end{equation}
Assume that $x_1\neq0$. We know that $x_1=y_2$, $x_2=y_1$ then we have from (\ref{2.19})  \\
 $\alpha=\frac{\sin t_2}{\sin t_1}=\pm1$, $x_2=y_1=0$ and from (\ref{2.17}) $u=\pm1$.

 In this case from  $\cos t_{1}=0$ it follows that $\cos t_{2}=0$ and solutions of the system (\ref{2.10}) are
 $x_1=y_2=\pm1$, $x_2=y_1=0$.
And we have the following isomorphic two-dimensional  algebras $A_{0}^{+}$ and $A_{0}^{-}$ with structural constants matrix

 $$
\mathcal M_{0}^{+}=\left (\begin{array}{cccc}
 0& 1\\[2mm]
 0&  -1
\end{array} \Bigg|
 \begin{array}{cccc}
-1&  0\\[2mm]
  1&  0
\end{array} \right) \qquad  \text{and} \ \ \
\mathcal M_{0}^{-}=\left (\begin{array}{cccc}
 0& -1\\[2mm]
 0&  1
\end{array} \Bigg|
 \begin{array}{cccc}
1&  0\\[2mm]
 -1&  0
\end{array} \right)
$$
respectively.

\begin{pro}
 $A_{0}^{+}$ is not isomorphic to $A_{1}$.
\end{pro}
\begin{proof}
Indeed if we assume that they are isomorphic then we have a system similar to (\ref{2.10}). And this system will be in the following form:
\begin{equation}\label{2.20}
\left\{\begin{array}{cccc}
x_{1}^{2} + x_{1}x_{2}=y_1\\[1mm]
x_{2}^{2} +x_{1}x_{2}=y_2\\[1mm]
x_{1}y_{1} + x_{1}y_{2}=-y_1\\[1mm]
x_{2}y_{1} + x_{2}y_{2}=-y_2\\[1mm]
x_{1}y_{1} + x_{2}y_{1}=-x_1  \\[1mm]
x_{1}y_{2} + x_{2}y_{2}=-x_2  \\[1mm]
y_{1}^{2} + y_{1}y_{2}=x_1 \\[1mm]
y_{2}^{2} +   y_{1}y_{2}=x_2  \\[1mm]
x_{1}y_{2}\neq x_{2}y_{1}.
\end{array}\right.
\end{equation}
If we add the  equations of the system then we obtain
\begin{equation}\label{2.21}
(x_1 + x_2 +y_1 + y_2)^2=0.
\end{equation}
We get the following system by adding the first and last two equations:
$$
\left\{\begin{array}{cccc}
(x_1 + x_2)^2=y_1 + y_2\\
(y_1 + y_2)^2=x_1 + x_2 .
\end{array}\right.
$$
From this system and (2.21) we have $x_1 + x_2 = y_1 + y_2=0$.
It is contradiction to $x_{1}y_{2}\neq x_{2}y_{1}$.
\end{proof}

\textbf{Case 1.2.2.2.} Let $\cos t_1\neq0$, then $\cos t_2\neq0$. From (\ref{2.18}) we obtain
\begin{equation}\label{2.22}
u=\frac{\cos t_2}{\cos t_1}.
\end{equation}
If we use the equality (\ref{2.22}) in the system (\ref{2.11}), then we have the system (\ref{2.19}) again.
In this case $\sin t_2\neq0$ otherwise from (\ref{2.19}) we get $x_1=x_2=y_1=y_2=0$. It is contradiction.
If we assume  $x_1\neq0$ then we have $\alpha=\frac{\sin t_2}{\sin t_1}$. In this case by system (\ref{2.11}) we have  that the system (2.10) also
has a unique solution  $x_1=y_2=\alpha=u\neq0$, $x_2=y_1=0$.
It means that  the equality
$$ \frac{\sin t_2}{\sin t_1}=\frac{\cos t_2}{\cos t_1}$$
 must be satisfied for the system (\ref{2.10}) to have a solution. Consequently $\sin (t_2 - t_1)=0$, i.e. $t_2=t_1 + \pi k, \ k\in Z.$

If $x_1=0$ then we can assume $x_2\neq0$. In this case $x_1=y_2=0$, $x_2=y_1=u=-\alpha$ and also $\sin(t_2-t_1)=0$.

 So, if $\sin t_1 \sin t_2 \cos t_1 \cos t_2\neq0$ then the system (2.10) has a solution\\
  $x_1=y_2=\frac{\cos t_2}{\cos t_1}$, $x_2=y_1=0$   or $x_1=y_2=0$, $x_2=y_1=\frac{\cos t_2}{\cos t_1}$, where $t_2=t_1 + \pi k, \ k\in Z.$

In this case we have the following  two-dimensional  algebras $A_{\cos t}^{+}$ and $A_{\cos t}^{-}$ with structural constants matrix
$$
\mathcal M_{\cos t}^{+}=\left (\begin{array}{cccc}
 \cos t& \sqrt{1-\cos^{2}t}\\[2mm]
 \cos t&  -\sqrt{1-\cos^{2}t}
\end{array} \Bigg|
 \begin{array}{cccc}
-\sqrt{1-\cos^{2}t}&  \cos t\\[2mm]
  \sqrt{1-\cos^{2}t}&  \cos t
\end{array} \right) \qquad \text{and}$$
$$\mathcal M_{\cos t}^{-}=\left (\begin{array}{cccc}
 \cos t& -\sqrt{1-\cos^{2}}t\\[2mm]
 \cos t&  \sqrt{1-\cos^{2}t}
\end{array} \Bigg|
 \begin{array}{cccc}
\sqrt{1-\cos^{2}t}&  \cos t\\[2mm]
  -\sqrt{1-\cos^{2}t}&  \cos t
\end{array} \right)
\qquad \text{respectively},$$
 where $\cos t\in(-1;0)\cup(0;1)$ and the sign above $\mathcal M_{\cos t}$ represents the sign of the $\sin t$ in the matrix (\ref{2.5}).

The following theorem answers the question of whether the algebras found are different algebras.
\begin{thm}\label{thm2}
(1) $A_{\cos t}^{+}$ is isomorphic to  $A_{-\cos t}^{-}$ for any $\cos t\in(-1;1)\setminus \{0\}$.

(2) $A_{\cos t_1}^{+}$ is not isomorphic to $A_{\cos t_2}^{+}$ for any $\cos t_1,\cos t_2\in(0;1), \cos t_1\neq \cos t_2. $

(3) $A_{\cos t_1}^{-}$ is not isomorphic to $A_{\cos t_2}^{-}$ for any $\cos t_1,\cos t_2\in(0;1), \cos t_1\neq \cos t_2. $

(4) $A_{\cos t_1}^{+}$ is not isomorphic to $A_{\cos t_2}^{-}$ for any $\cos t_1,\cos t_2\in(0;1)$.
\end{thm}
\begin{proof} (1)
Let $\{e_{1},  e_{2}\}$ and $\{e_{1}^{*},  e_{2}^{*}\}$ be bases of $A_{\cos t}^{+}$ and  $A_{-\cos t}^{-}$ respectively.
If we take the  change  $e_{1}^{*}=-e_{1}$, $e_{2}^{*}=-e_{2}$ of basis then it is easy to see that they are isomorphic.

(2) Assume that $A_{\cos t_1}^{+}$ is isomorphic to $A_{\cos t_2}^{+}$ for some $\cos t_1,\cos t_2\in(0;1), \cos t_1\neq \cos t_2 $.
If $\{e_{1},  e_{2}\}$, $\{e_{1}^{*},  e_{2}^{*}\}$ be bases of $A_{\cos t_1}^{+}$ and  $A_{\cos t_2}^{+}$ respectively
then from above-mensioned results (see Case 1.2.2.2) it is easy to see that there are only the following changes of basis:
$$
(a) \left\{\begin{array}{cc} e_1^{*}=e_1\\ e_2^{*}=e_2 \end{array} \right.,
(b) \left\{\begin{array}{cc} e_1^{*}=-e_1\\ e_2^{*}=-e_2 \end{array} \right.,
(c) \left\{\begin{array}{cc} e_1^{*}=e_2\\ e_2^{*}=e_1 \end{array} \right.,
(d) \left\{\begin{array}{cc} e_1^{*}=-e_2\\ e_2^{*}=-e_1 \end{array} \right..
$$
It is not difficult to check that if we take the change of basis $(a)$ or $(c)$ then  it follows that $\cos t_1=\cos t_2$.
If we take the change of basis $(b)$ or $(d)$ then it follows that  $\cos t_1 = -\cos t_2$. It is contradiction.

(3) The proofs of (3) and (4) are  similar to the proof of (2).
\end {proof}
The first part of Theorem \ref{thm2}  means that $A_{\cos t}^{+}$ with  positive $\cos t$ (with negative $\cos t$)  is isomorphic to $A_{\cos t}^{-}$
 with negative $\cos t$ (with positive $\cos t$).

 \textbf{Case 2.}  Let  $\cos t_{2}=- \sin t_{2}$. In this case $ A^{[t_{2}]}$ is commutative.
 According to our assumption $ A^{[t_{1}]}$ is also commutative. It means that $\cos t_{1}=- \sin t_{1}$.
 Therefore in this case  $t_1, t_2 \in \{\frac{3\pi}{4}+\pi k, \  k=0,1,2,...  \}$. The possible cases in this case are as follows\\
 $$(a)\  \left\{\begin{array}{cccc}
 t_1=\frac{3\pi}{4}\\[1mm] t_2=\frac{3\pi}{4}\end{array}\right.,
  \ \ (b)\  \left\{\begin{array}{cccc}
 t_1=\frac{3\pi}{4}\\[1mm] t_2=\frac{7\pi}{4}
\end{array}\right.,
\ \ (c)\  \left\{\begin{array}{cccc}
 t_1=\frac{7\pi}{4}\\[1mm] t_2=\frac{7\pi}{4}
\end{array}\right.,
\ \ (d)\  \left\{\begin{array}{cccc}
 t_1=\frac{7\pi}{4}\\[1mm] t_2=\frac{3\pi}{4}
\end{array}\right..$$
For the cases $(a)$ and $(c)$: $\left\{\begin{array}{cccc} \cos t_1=\cos t_2\\ \sin t_1=\sin t_2 \end{array}\right.$
and for the cases $(b)$ and $(d)$: $\left\{\begin{array}{cccc} \cos t_1=-\cos t_2\\ \sin t_1=-\sin t_2 \end{array}\right.$.
 In this case from  $\cos t_{2}=-\sin t_2$ it follows that $\cos t_{1}=-\sin t_1$ and one of the solutions of the system (\ref{2.10}) is
 $x_1=y_2=1$, $x_2=y_1=0$ or $x_1=y_2=-1$, $x_2=y_1=0$.
And we have the following  two-dimensional commutative, isomorphic algebras $A_{-2}$ and $A_{2}$ with structural constants matrix
 $$
\mathcal M_{-2}=\left (\begin{array}{cccc}
 -\frac{\sqrt{2}}{2}&  \frac{\sqrt{2}}{2}\\[2mm]
  -\frac{\sqrt{2}}{2}&   -\frac{\sqrt{2}}{2}
\end{array} \Bigg|
 \begin{array}{cccc}
 -\frac{\sqrt{2}}{2}&   -\frac{\sqrt{2}}{2}\\[2mm]
   \frac{\sqrt{2}}{2}&   -\frac{\sqrt{2}}{2}
\end{array} \right) \qquad \text{and} \ \ \
\mathcal M_{2}=\left (\begin{array}{cccc}
  \frac{\sqrt{2}}{2}& -\frac{\sqrt{2}}{2}\\[2mm]
 \frac{\sqrt{2}}{2}&  \frac{\sqrt{2}}{2}
\end{array} \Bigg|
 \begin{array}{cccc}
\frac{\sqrt{2}}{2}&  \frac{\sqrt{2}}{2}\\[2mm]
 -\frac{\sqrt{2}}{2}&  \frac{\sqrt{2}}{2}
\end{array} \right)
$$
respectively.

From the above results we obtain the following theorem.
In this  theorem we give a time-depending classification of time-homogeneous FA $A^{[t]}$ corresponding to the
matrix (\ref{2.5}).
\begin{thm}\label{thm3}
$$A^{[t]} \cong \left\{\begin{array}{cccc}
A_{1},& if & t\in \{ \pi k: k\in I\}\\[2mm]
A_{0}^{+},& if & t\in \{\frac{\pi}{2}+ \pi k: k\in I\}\\[2mm]
A_{2},& if & t\in \{\frac{3\pi}{4}+ \pi k: k\in I\}\\[2mm]
A_{\cos t}^{+},&  if & t\in \bigcup_{k\in I} ( \pi k; \frac{\pi}{2}+\pi k) \\[2mm]
A_{\cos t}^{-},&  if & t\in \bigcup_{k\in I} (( \frac{\pi}{2}+\pi k; \pi + \pi k)\setminus \{\frac{3\pi}{4}+ \pi k\})
 \end{array}\right.$$
 where $ I=\{0,1,2,...\}$ is set of nonnegative integers and $\cos t \in (0;1)$.
\end{thm}
We can better see the essence of the theorem using the figure below
\begin{figure}[h!]
\includegraphics[width=0.95\textwidth]{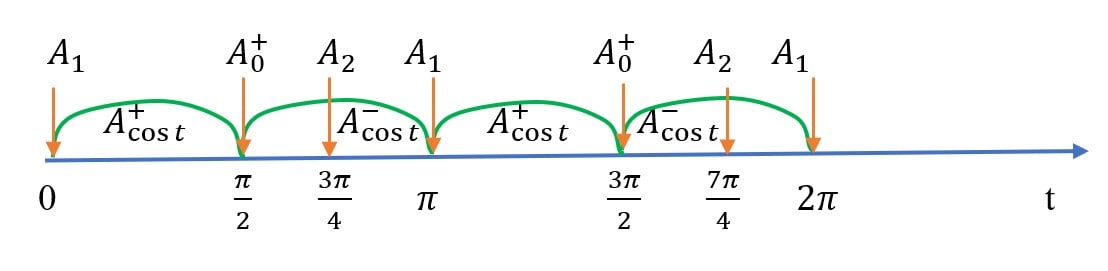}\\
\caption{The partition of the time set $\{(0,t): 0\leq t\}$ corresponding to the classification of algebras in the FA $A^{[t]}$ with the matrix (\ref{2.5}).}
\end{figure}

It is known \cite{Jacobson} that a finite-dimensional algebra $A$ with a matrix of structural constants $(c_{ijk})_{i,j,k=1}^{m}$
is associative iff
\begin{equation}\label{2.23}
\sum_{r=1}^m{c_{ijr}c_{rkl}}=\sum_{r=1}^m{c_{irl}c_{jkr}}, \ \ \ \ \ \mathrm{for\  all}\ i,j,k,l.
\end{equation}
Now we check the associativity condition for FA $A^{[t]}$ with the matrix (\ref{2.5}).
For $m=2$ the system (\ref{2.23}) will be the following form
$$
c_{ij1}c_{1kl} + c_{ij2}c_{2kl}=c_{i1l}c_{jk1} + c_{i2l}c_{jk2} \ \ \ \ \ i,j,k,l=1,2.
$$
  After some easy calculations we have the following remark.
\begin{rk}
Among the algebras given in Theorem \ref{thm3}, only $A_1$ and $A_2$ are associative. And note that  $A_2$ is associative and commutative algebra.
\end{rk}
 Does this FA include evolution algebra? Before answering this question, we recall some notions. \begin{defn} \cite{Tian}
  Let $(E,\cdot)$ be an  algebra  over  a  field  $K$. If it admits a basis $e_1,e_2...,$ such that
 $$e_{i}\cdot e_{j}=0,  \ \ \mbox{if} \ \  i\neq j \ ; $$
$$
\ e_{i}\cdot e_{i}=\sum\limits_{k} a_{ik}e_k, \ \
 \mbox{for any} \ \  i,
$$
 then this algebra is called an $evolution \  algebra$. This basis is  called natural basis.
\end{defn}
Note that an $n$-dimensional evolution algebra is a non-associative, commutative algebra. In \cite{Velasco}, necessary and sufficient conditions for a given commutative algebra to be an evolution algebra were found.

The flow we considered has only one commutative algebra $A_2$, and it is associative.
\begin{rk}
The FA $A^{[t]}$ corresponding to the matrix (\ref{2.5}) does not include any evolution algebra.
\end{rk}

\section{\textbf{Comparision with Theorem 1}}

\begin{thm}
The algebras mentioned in Theorem \ref{thm3} have the following relations with classes of Theorem \ref{thm1}:
 $$ A_1\cong \mathcal{A}_{5}(\frac{1}{2}, 0),\ \  A_0^{+}\cong  \mathcal{A}_{8}(0, 0), \ \ A_2\cong \mathcal{A}_{3}(\frac{1}{2}, 0,\frac{1}{2}),$$
  $$ A_{\cos t}^{+}\cong  \mathcal{A}_{2}(\frac{1}{2}, 0, \frac{- \sin t}{2 \cos t}), \ \ A_{\cos t}^{-} \cong  \mathcal{A}_{3}(\frac{1}{2}, 0, \frac{\sin t}{2 \cos t}).$$
\end{thm}
\begin{proof}
If we write the matrix (\ref{2.2}) in the form of the matrix (\ref{1.2}), then the  matrix of structural constants
 (\ref{2.5}) of the time-homogeneous FA will be the following
\begin{equation}\label{3.1}
\mathcal M^{[t]}_e=\left (\begin{array}{cccc}
 \cos t& \cos t& -\sin t& \sin t\\[2mm]
  \sin t& - \sin t& \cos t& \cos t
\end{array} \right).
\end{equation}
First we consider the algebra $A_{\cos t}^{+}$. Now we use the change of basis as in \cite{Bekbaev}.
If we take a change of basis
$$e_1^*=\frac{1}{4\cos t}(e_1 + e_2), \ e_2^*=\frac{1}{2\sqrt{\sin 2t}}(e_1 - e_2)$$
 then the matrix of structural
constants of $A_{\cos t}^{+}$  will be the following
\begin{equation}\label{3.2}
\mathcal M^{+}_{\cos t, e^*}=\left (\begin{array}{cccc}
 \frac{1}{2}& 0&   0& 1 \\[2mm]
  0&  \frac{- \sin t}{2 \cos t}& \frac{1}{2}& 0
\end{array} \right).
\end{equation}
where $\sin t, \cos t \in (0;1)$.

It means that $A_{\cos t}^{+}$ is isomorphic to $\mathcal{A}_{2}(\frac{1}{2}, 0, \frac{- \sin t}{2 \cos t})$.

Similar to the above case we can show that
$A_{\cos t}^{-}$ is isomorphic to $\mathcal{A}_{3}(\frac{1}{2}, 0, \frac{\sin t}{2 \cos t})$.
Consequently $A_2$ is isomorphic to $\mathcal{A}_{3}(\frac{1}{2}, 0,\frac{1}{2})$.

Now we consider the algebras $A_1$ and $A_0^{+}$. We write their matrices of structural constants similar to (\ref{1.2}).
Then we get the following matrices respectively
\begin{equation}\label{3.4}
\mathcal M_{1,e}=\left (\begin{array}{cccc}
 1& 1&   0& 0\\[2mm]
  0&  0& 1& 1
\end{array} \right)
\end{equation}
and
\begin{equation}\label{3.5}
\mathcal M_{0,e}^{+}=\left (\begin{array}{cccc}
 0& 0&   -1& 1\\[2mm]
  1&  -1& 0& 0
\end{array} \right).
\end{equation}
If we take a change $e^*_1=\frac{1}{2}e_1, \ e^*_2=-e_1 + e_2$ for (\ref{3.4}) then we get
 $$
 \mathcal M_{1,e^*}=\left (\begin{array}{cccc}
 \frac{1}{2}& 0&   0& 0\\[2mm]
  0&  0&  \frac{1}{2}& 0
\end{array} \right).
$$
It means that   $A_1$ is isomorphic to $\mathcal{A}_{5}(\frac{1}{2}, 0)$.

If we take a change
$e^*_1=-\frac{1}{2}(e_1+e_2),\  e^*_2=\frac{1}{2}(e_1 - e_2)$
for (\ref{3.5}) then we get
$$
\mathcal M_{0,e^{*}}^{+}=\left (\begin{array}{cccc}
 0&  0&   0& -1\\[2mm]
  0&  1&  0& 0
\end{array} \right).
$$
Hence, the algebra $A_0^{+}$ is isomorphic to $\mathcal{A}_{8}(0, 0)$.
\end{proof}

\section*{Acknowledgements}
The authors thank Dr. U. Bekbaev for helpful suggestions.

 Rozikov thanks Weierstrass Institute for Applied Analysis and Stochastics (Berlin, Germany) for an invitation and support, his work was partially supported by a grant from the IMU-CDC.

Velasco was supported by Junta de Andaluc\'{i}a grant FQM-199, and  the
Spanish Ministry of Science and Innovation (MINECO/MICINN/FEDER),
through the IMAG-Maria de Maeztu Excellence Grant
CEX2020-001105M/AEI/10.13039/501100011033// CEX2020-001105-M.

Narkuziev thanks Granada university for an invitation and Agency for Innovative Development under the Ministry of higher education, science and innovations  of the Republic of Uzbekistan for support.

\end{document}